\documentclass[12pt]{amsart}

\usepackage[left=1in, right=1in, top=1in, bottom=1in]
{geometry}
\usepackage{amsfonts,amssymb,graphicx,amsmath,amsthm, float}
\usepackage{xypic}
\usepackage[all]{xy}

\theoremstyle{plain}
\newtheorem{theorem}{Theorem}[section]

\newtheorem{lemma}[theorem]{Lemma}
\newtheorem{example}[theorem]{Example}

\newtheorem{corollary}[theorem]{Corollary}

\newtheorem{proposition}[theorem]{Proposition}

\newcommand{\N}{\mathbb{N}}

\theoremstyle{remark}

\newcommand{\Z}{\mathbb{Z}}

\begin{document}

\title{Generalized Factorization in Commutative Rings with Zero-Divisors}
        \date{\today}
\author{Christopher Park Mooney}
\address{Reinhart Center \\ Viterbo University \\ 900 Viterbo Drive \\ La Crosse, WI 54601}
\email{cpmooney@viterbo.edu}


\begin{abstract}
Much work has been done on generalized factorization techniques in integral domains, namely $\tau$-factorization.  There has also been substantial progress made in investigating factorization in commutative rings with zero-divisors.  This paper seeks to synthesize work done in these two areas and extend the notion of $\tau$-factorization to commutative rings that need not be domains.  In addition, we look into particular types of $\tau$ relations, which are interesting when there are zero-divisors present.  We then proceed to classify commutative rings that satisfy the finite factorization properties given in this paper. \\
\vspace{.1in}\noindent \textbf{2010 AMS Subject Classification:} 13A05, 13E99, 13F15\end{abstract}

\maketitle
\section{Introduction}
This paper concerns generalized factorization in a commutative ring $R$ with $1$.  Much work has been done on generalized factorization techniques in integral domains;  D.D. Anderson and A. Frazier provide an excellent overview in \cite{Frazier}.  Many authors have investigated ways to extend notions of factorization in domains to commutative rings with zero-divisors. This list includes, but is not limited to \cite{Chun, Valdezleon, Valdezleon3, Agargun, Axtell, Fletcher, Fletcher2}.  The goal of this paper is to extend $\tau$-factorization to rings that contain zero-divisors.
\\
\indent In Section Two, we give some preliminary definitions and results about both $\tau$-factorization in integral domains as well as factorization in commutative rings with zero-divisors.  In Section Three, we define several types of $\tau$-irreducible elements.  Section Four studies various finite $\tau$-factorization properties.  In Section Five, we introduce a particular $\tau$ relation which is natural in rings with zero-divisors.  Furthermore, we provide a thorough examination of rings which satisfy the various $\tau$-finite factorization properties laid forth in Section Four.

\section{Preliminary Definitions and Results}
\indent We let $R$ be a commutative ring with identity.  Let $R^*=R-\{0\}$, $U(R)$ be the units of $R$, and $R^{\#}=R^*-U(R)$, the non-zero, non-units of $R$.  As in \cite{Valdezleon}, we let $a \sim b$ if $(a)=(b)$, $a\approx b$ if there exists $\lambda \in U(R)$ such that $a=\lambda b$, and $a\cong b$ if (1) $a\sim b$ and (2) $a=b=0$ or if $a=rb$ for some $r\in R$ then $r\in U(R)$.
\\
\indent Let $\tau$ be a relation on $R^{\#}$, that is, $\tau \subseteq R^{\#} \times R^{\#}$.  In this paper, we will always assume that $\tau$ is symmetric.  We say $\tau$ is \emph{multiplicative} (resp. \emph{divisive}) if for $a,b,c \in R^{\#}$ (resp. $a,b,b' \in R^{\#}$), $a\tau b$ and $a\tau c$ imply $a\tau bc$ (resp. $a\tau b$ and $b'\mid b$ imply $a \tau b'$).  We say $\tau$ is \emph{associate} (resp. \emph{strongly associate}, \emph{very strongly associate) preserving} if for $a,b,b'\in R^{\#}$ with $b\sim b'$ (resp. $b\approx b'$, $b\cong b'$) $a\tau b$ implies $a\tau b'$.  As in \cite{Stickles}, a ring $R$ is said to be \emph{strongly associate} (resp. \emph{very strongly associate}) ring if for any $a,b \in R$, $a\sim b$ implies $a \approx b$ (resp. $a \cong b$).
\\
\indent For a non-unit $a \in R$, we define $a=\lambda a_1 \cdots a_n$, $\lambda \in U(R)$, $a_i \in R^{\#}$ to be a \emph{$\tau$-factorization of} $a$ if $a_i\tau a_j$ for each $i\neq j$.  We call $a=\lambda(\lambda^{-1}a)$ a \emph{trivial $\tau$-factorization of $a$}.  We say that $a$ is a \emph{$\tau$-product} of the $a_i$'s and that $a_i$ is a $\tau$-\emph{factor} or a $\tau$-\emph{divisor} of $a$.  We do not allow $0$ to occur as a $\tau$-factor of a non-trivial $\tau$-factorization; however, we do allow the trivial factorization, $0=\lambda 0$ for $\lambda \in U(R)$.  For $a,b \in R^{\#}$ we say that $a$ \emph{$\tau$-divides $b$}, written $a\mid_{\tau}b$, if $a$ occurs as a $\tau$-factor in some $\tau$-factorization of $b$.  Note that if $a=\lambda a_1 \cdots a_n$ is a $\tau$-factorization, then for $\sigma \in S_n$, the symmetric group on $n$ letters, so is each rearrangement of $a=\lambda a_{\sigma(1)}\cdots a_{\sigma(n)}$ because $\tau$ is assumed to be symmetric.
\\
\indent A \emph{$\tau$-refinement} of a $\tau$-factorization $\lambda a_1 \cdots a_n$ is a $\tau$-factorization of the form
$$\lambda \cdot b_{1_1}\cdots b_{1_{m_1}}\cdot b_{2_1}\cdots b_{2_{m_2}} \cdots b_{n_1} \cdots b_{n_{m_n}}$$
where $a_i=b_{i_1}\cdots b_{i_{m_i}}$ is a $\tau$-factorization for each $i$.  We say that $\tau$ is \emph{refinable} if every $\tau$-refinement of a $\tau$-factorization is a $\tau$-factorization.  We say $\tau$ is \emph{combinable} if whenever $\lambda a_1 \cdots a_n$ is a $\tau$-factorization, then so is each $\lambda a_1 \cdots a_{i-1}(a_ia_{i+1})a_{i+2}\cdots a_n$.
\\
\indent We pause briefly to give some examples of particular relations $\tau$.
\begin{example} \end{example}
\indent(1) $\tau=R^{\#}\times R^{\#}$.  This yields the usual factorizations in $R$ and $\mid_{\tau}$ is the same as the usual divides. Moreover, $\tau$ is multiplicative and divisive (hence associate preserving as we shall soon see).  This case shows $\tau$-factorization is a generalization of the usual factorization in commutative rings with zero-divisors.
\\ \indent(2) $\tau=\emptyset$.  For every $a\in R^{\#}$, there is only the trivial factorization and $a\mid{_\tau} b \Leftrightarrow a=\lambda b$ for $\lambda \in U(R)$ $\Leftrightarrow a\approx b$.  Again $\tau$ is both multiplicative and divisive (vacuously).
\\ \indent(3) Let $S$ be a non-empty subset of $R^{\#}$ and let $\tau=S\times S$.  Define $a\tau b \Leftrightarrow a,b\in S$. So $\tau$ is multiplicative (resp. divisive) if and only if $S$ is multiplicatively closed (resp. closed under non-unit factors).  A non-trivial $\tau$-factorization is (up to unit factors) a factorization into elements from $S$.
\\ \indent(4) Let $I$ be an ideal of $R$ and define $a \tau b$ if and only if $a-b\in I$.  This relation is certainly symmetric, but need not be multiplicative or divisive.  Let $R=\Z$ and $I=(5)$.  We have $7\tau 2$ and $7 \tau 7$, but $7 \not \tau 14$, showing $\tau$ is not multiplicative.  Moreover, $9 \tau 4$ and $2 \mid 4$; however, $9\not \tau 2$ showing $\tau$ is not divisive.
\\ \indent(5) Let $a \tau_z b \Leftrightarrow ab=0$.  The only non-trivial $\tau_z$-factorizations are $0=\lambda a_1 \cdots a_n$ where $a_i \cdot a_j=0$ for all $i \neq j$. Thus $a\in R^{\#}$ has only the trivial $\tau_z$-factorization.  This particular $\tau$ will be studied in Section Five.

\begin{lemma} \label{lem: relations}Let $R$ be a commutative ring and let $a,b\in R^{\#}$.
\\
(1) $a\cong b \Rightarrow a\approx b \Rightarrow a\sim b$.
\\
(2) $\sim$ and $\approx$ are equivalence relations.
\\
(3) $\cong$ need only be transitive and symmetric.
\\
(4) For $a\in R$, the following are equivalent.
\\\indent(a) $a\sim b$ for some $b\in R$ implies $a\cong b$.
\\\indent(b) $a \cong a$.
\\\indent(c) $a=0$ or $\text{ann}(a) \subseteq J(R)$.
\\ If $a$ satisfies one of the above conditions, then for $\lambda \in U(R)$, $a\cong b$ $\Leftrightarrow$ $a\cong \lambda b$.
\\(5) The following conditions are equivalent.
\\\indent(a) $R$ is very strongly associate.
\\\indent(b) $R$ is pr\'esimplifiable (for all $x,y \in R$, $xy=x$ implies $x=0$ or $y\in U(R)$).
\\\indent(c) $\cong$ is reflexive on $R$.
\\\indent(d) $\cong$ is an equivalence relation on $R$.
\\\indent(e) $\sim$, $\approx$, and $\cong$ all coincide on $R$.
\\ In particular, domains and quasi-local rings all satisfy the above conditions.
\end{lemma}
\begin{proof}
See \cite[Theorem 2.2]{Valdezleon} and the discussion preceeding the theorem.
\end{proof}

The following theorem is a slight generalization of \cite[Proposition 2.2]{Frazier}.
\begin{theorem} Let $R$ be a commutative ring and $\tau$ a relation on $R^{\#}$. Let $a,b,b' \in R^{\#}$, $\lambda \in U(R)$.
\\ (1) If $\tau$ is divisive, then $\tau$ is associate (resp. strongly associate, very strongly associate) preserving.
\\ (2) If $\tau$ is divisive, then $\tau$ is refinable.
\\ (3) If $\tau$ is multiplicative, then $\tau$ is combinable.
\end{theorem}
\begin{proof}
(1) Let $a\tau b$. Now $b\sim b' \Rightarrow b\mid b'$ and $b'\mid b$.  So by the definition of divisive, $a\tau b \Rightarrow a\tau b'$ and $a \tau b' \Rightarrow a\tau b$.  As $b\cong b'$ and $b\approx b'$ each imply $b\sim b'$ by Lemma \ref{lem: relations}, the result follows.  Proofs of (2) and (3) can be found in \cite{Frazier}.
\end{proof}
\section{Types of $\tau$-Irreducible Elements}
\label{sec: types of tau-irreducibles}
We would like to define what it means for an element to be $\tau$-irreducible in a ring with zero-divisors.  This definition needs to be consistent with the definitions of $\tau$-irreducible in domains as well as the various types of irreducible elements when zero-divisors are present.  These definitions are generalizations of those given in \cite{Valdezleon}.
\begin{proposition}\label{prop: irr} Let $R$ be a commutative ring and $\tau$ a relation on $R^{\#}$ with $a\in R$ a non-unit and $\lambda \in U(R)$.  Consider the following statements.
\\(1) $a=\lambda a_1 \cdots a_n$, $n\in \N$, is a $\tau$-factorization implies $a \sim a_i$ for some $1\leq i \leq n$.
\\(2) $(a)=(a_1)\cdots (a_n)$, $n\in \N$, with $a_i \tau a_j$ for all $i \neq j$ implies $(a)=(a_i)$ for some $1\leq i \leq n$.
\\(3) $a\sim a_1 \cdots a_n$, $n\in \N$, with $a_i\tau a_j$ for all $i \neq j$ implies $a \sim a_i$ for some $1\leq i \leq n$.
\\ We have (2) $\Leftrightarrow$ (3) $\Rightarrow$ (1).  If $R$ is strongly associate, we also have (1) $\Rightarrow$ (2).
\end{proposition}
\begin{proof}
(2) $\Leftrightarrow$ (3) are seen to be equivalent after noting $(a_1)(a_2)\cdots (a_n)=(a_1\cdot a_2 \cdots a_n)$.
\\
\indent (2) $\Rightarrow$ (1) If $a=\lambda a_1 \cdots a_n$ is a $\tau$-factorization, then $(a)=(a_1\cdots a_n) = (a_1)\cdots(a_n)$ with $a_i \tau a_j$ for all $i \neq j$, so by (2) we have $(a)=(a_i)$ for some $1\leq i \leq n$.
\\
\indent We now assume $R$ is strongly associate and show $(1) \Rightarrow (2)$. Let $(a)=(a_1) \cdots (a_n)$ with $a_i\tau a_j$ for $i \neq j$, then $a\sim a_1 \cdots a_n$ implies there exists a $\lambda \in U(R)$ with $a=\lambda a_1 \cdots a_n$ a $\tau$-factorization, so by (1) we have $a\sim a_i$ for some $i$.
\end{proof}

\indent We will call a non-unit $a \in R$ \emph{$\tau$-irreducible} or \emph{$\tau$-atomic} if it satisfies condition (1) of Proposition \ref{prop: irr}.
\begin{theorem} A strong associate of a $\tau$-irreducible element is $\tau$-irreducible.
\end{theorem}
\begin{proof} Let $a$ be a $\tau$-irreducible element.  Suppose $a=\lambda a'$ for $\lambda \in U(R)$.  Let $a'= \mu b_1 \cdots b_n$ be a $\tau$-factorization.  Then $a=\lambda a'= (\lambda \mu) b_1 \cdots b_n$ is a $\tau$-factorization and $a$ is $\tau$-irreducible, so $a\sim b_i$ for some $1 \leq i \leq n$.  We have $a \approx a' \Rightarrow a\sim a'$, so $a' \sim a\sim b_i$ which shows $a'$ is $\tau$-irreducible.
\end{proof}
\begin{proposition}\label{prop: s irr}Let $R$ be a commutative ring and $\tau$ a relation on $R^{\#}$. For $a\in R$, a non-unit and $\lambda \in U(R)$, the following are equivalent.
\\(1) $a=\lambda a_1\cdots a_n$, $n\in \N$, with $a_i \tau a_j$ for $i \neq j$ implies $a \approx a_i$ for some $i$.
\\(2) $a\approx a_1\cdots a_n$, $n\in \N$, with $a_i \tau a_j$ for $i \neq j$ implies $a \approx a_i$ for some $i$.
\end{proposition}
\begin{proof}
This is immediate from definitions.
\end{proof}
We will call a non-unit element $a\in R$ \emph{$\tau$-strongly irreducible} or \emph{$\tau$-strongly atomic} if $a$ satisfies one of the conditions of Proposition \ref{prop: s irr}.
\begin{theorem} A strong associate of a $\tau$-strongly irreducible element is $\tau$-strongly irreducible.
\end{theorem}
\begin{proof}Let $a' \approx a$ with $a$ $\tau$-strongly irreducible.  Suppose $a'\approx a_1 \cdots a_n$, $n\in \N$, with $a_i\tau a_j$ for all $i \neq j$.  Then we have $a\approx a' \approx a_1 \cdots a_n$ which implies $a \approx a_i$ for some $1\leq i \leq n$.  Hence $a'\approx a \approx a_i$, showing $a'$ to be $\tau$-strongly irreducible.
\end{proof}

\begin{proposition} \label {prop: m irr} Let $R$ be a commutative ring and $\tau$ a relation on $R^{\#}$.  For a non-unit $a \in R$, $\lambda \in U(R)$ we consider the following statements.
\\(1) $(a)$ is maximal in the set $S':=\{(b) \mid b\in R$, a non-unit and $b\mid_{\tau} a \}$.
\\(2) $a=\lambda a_1 \cdots a_n$, a $\tau$-factorization implies $a \sim a_i$ for all $i$.
\\(3) $a=\lambda a_1 \cdots a_n$, a $\tau$-factorization implies $a \approx a_i$ for all $i$.
\\Then (3) $\Rightarrow$ (1) $\Leftrightarrow$ (2) and for $R$ strongly associate, $(2) \Rightarrow (3)$.
\end{proposition}
\begin{proof}
(1) $\Rightarrow$ (2) Let $a$ satisfy (1) and suppose $a=\lambda a_1 \cdots a_n$ is a $\tau$-factorization.  Then $(a)\subseteq (a_i)$, each $a_i\mid_\tau a$, so we have $(a_i)\in S'$ for all $i$.  Hence by maximality of $(a)$ in $S'$, we have $(a)=(a_i)$ as desired.  (2) $\Rightarrow$ (1) Suppose $a$ satisfies (2), and we have $(a) \subseteq (b) \in S'$.  We have $b\mid_{\tau}a$.  Say $a=\lambda ba_1 \cdots a_n$ is a $\tau$-factorization.  By (2) we have $a\sim b$, thus proving $(a)$ is maximal in $S'$ as desired.
\\
\indent (3) $\Rightarrow$ (2) Clear. Furthermore, given $R$  strongly associate it is clear that the converse will also hold since $a\sim a_i \Rightarrow a\approx a_i$.
\end{proof}
We say a non-unit element $a \in R$ is \emph{$\tau$-m-irreducible} or \emph{$\tau$-m-atomic} if $a$ satisfies conditions (1) or (2) of Proposition \ref{prop: m irr}.
\begin{theorem}
A strong associate of a $\tau$-m-irreducible element is $\tau$-m-irreducible.
\end{theorem}
\begin{proof}
Let $a$ be a $\tau$-m-irreducible element.  Suppose $a'\approx a$.  Say there is a unit $\mu$ in $R$ with $a=\mu a'$.  We suppose $a'=\lambda a_1 \cdots a_n$, $n\in \N$, with $a_i\tau a_j$ for all $i \neq j$.  Then $a=\mu a'=(\mu \lambda) a_1 \cdots a_n$ remains a $\tau$-factorization.  So by (2), $a \sim a_i$ for all $1\leq i \leq n$.  But then we have $a' \sim a \sim a_i$ for all non-units $a_i\in R$, showing $a$ is $\tau$-m-irreducible as desired.
\end{proof}
\begin{proposition}\label{prop: vs irr}
Let $R$ be a commutative ring and $\tau$ a relation on $R^{\#}$.  For a non-unit $a \in R$, $\lambda \in U(R)$, with $a\cong a$, the following are equivalent.
\\(1) $a=\lambda a_1 \cdots a_n$, $n\in \N$, with $a_i \tau a_j$ for all $i\neq j$ implies $a\cong a_i$ for some $i$.
\\(2) $a\cong a_1 \cdots a_n$, $n\in \N$, with $a_i \tau a_j$ for all $i\neq j$ implies $a\cong a_i$ for some $i$.
\\(3) $a \sim a_1 \cdots a_n$, $n\in \N$, with $a_i \tau a_j$ for $i \neq j$ implies $a \sim a_i$ for some $i$ .
\\(4) $a$ has no non-trivial $\tau$-factorizations.
\end{proposition}
\begin{proof}
\indent (1) $\Rightarrow$ (2) Suppose $a\cong a_1 \cdots a_n$ with $a_i \tau a_j$ for all $i\neq j$.  We have $a=\lambda a_1 \cdots a_n$ for some $\lambda \in U(R)$.  Thus by (1) $a\cong a_i$ for some $i$.  (2) $\Rightarrow$ (3)  Suppose $a\sim a_1 \cdots a_n$ with $a_i \tau a_j$ for all $i\neq j$. We have $a\cong a$, so $a\cong a_1 \cdots a_n$.  By (2) $a\cong a_i$ for some $i$.  (3) $\Rightarrow$ (1) Suppose $a= \lambda a_1 \cdots a_n$ with $a_i \tau a_j$ for all $i\neq j$.  Then $a\sim a_1 \cdots a_n$.  By (3) we have $a\sim a_i$ for some $i$.  Thus we have $a\cong a_i$ for some $i$, proving the equivalence of (1)-(3).
\\
\indent (1) $\Rightarrow$ (4) Suppose $a=\lambda a_1 \cdots a_n$. By assumption $a \cong a_i$ for some $i$, say $a_i= \mu a$ for $\mu \in U(R)$.  This factorization can be written as $a=\lambda a_1 \cdots \hat{a_i} \cdots a_n \cdot (\mu a)$.  But $a \cong a$, which means $a_1 \cdots \hat{a_i} \cdots a_n=\lambda'\in U(R)$, so $n=1$ and we have the trivial factorization $a=\lambda'(\mu a)$ after all.  (4) $\Rightarrow$ (1)  The only types of $\tau$-factorizations are the trivial ones, $a=\lambda (\lambda^{-1} a)$ and we have by assumption $a\cong a$, and by Lemma \ref{lem: relations}, $a\cong \lambda^{-1}a$.
\end{proof}
We shall call a non-unit $a\in R$ with $a\cong a$ \emph{$\tau$-very strongly irreducible} or \emph{$\tau$-very strongly atomic} if it satisfies one of the equivalent conditions (1)-(4) of Proposition \ref{prop: vs irr}.
\begin{theorem} A strong associate of a $\tau$-very strongly irreducible element is $\tau$-very strongly irreducible. \end{theorem}
\begin{proof}
Let $a$ be $\tau$-very strongly irreducible.  Let $a\approx a'$, say $a=\mu a'$ for some $\mu \in U(R)$.  Then $a\cong a$ if and only if $a' \cong a'$ by Lemma \ref{lem: relations}.  We suppose $a'=\lambda a_1\cdots a_n$, $n\in \N$, with $a_i \tau a_j$ for all $i\neq j$ and $\lambda \in U(R)$.  So we have $a=\mu a'=(\mu\lambda) a_1\cdots a_n$ which remains a $\tau$-factorization.  Since $a$ is $\tau$-very strongly irreducible, we have $a \cong a_i$ for some $i$.  So $a'\cong a \cong a_i$ showing $a'$ is $\tau$-very strongly irreducible.
\end{proof}
\begin{theorem} Let $R$ be a commutative ring and $\tau$ a relation on $R^{\#}$.  Let $a \in R$ be a non-unit.
\\(1) $a$ is $\tau$-very strongly irreducible implies $a$ is $\tau$-m-irreducible.
\\(2) For $R$ strongly associate, $a$ is $\tau$-m-irreducible implies $a$ is $\tau$-strongly irreducible.
\\(3) $a$ $\tau$-strongly irreducible implies $a$ is $\tau$-irreducible.
\\(4) $a$ $\tau$-very strongly irreducible implies $a$ is $\tau$-strongly irreducible.
\\(5) $a$ $\tau$-m-irreducible implies $a$ is $\tau$-irreducible.
\\
The following diagram summarizes our result ($\dagger$ represents a strongly associate ring):
$$\xymatrix{
\tau\text{-very strongly irred.}\ar@{=>}[dr] \ar@{=>}[r]& \tau\text{-strongly irred.} \ar@{=>}[r]& \tau \text{-irred.}\\
 & \tau\text{-m-irred.}\ar@{=>}[u]_{\dagger}\ar@{=>}[ur]  & &}$$
\end{theorem}
\begin{proof}
(1) Let $a$ be $\tau$-very strong irreducible, and suppose $(a) \subseteq (a_i) \in S'$.  The only $\tau$-factorizations of $a$ are trivial ones.  We must have $a=\lambda (\lambda^{-1}a)=\lambda a_i$, that is $a \approx a_i$ and thus $(a)=(a_i)$, proving $a$ is $\tau$-m-irreducible.
\\
\indent(2) Let $R$ be a strongly associate ring, with $a$, a $\tau$-m-irreducible element.  We suppose $a=\lambda a_1 \cdots a_n$ is a $\tau$-factorization.  Then $a_i \mid_{\tau} a$ for each $i$.  But $a$ is $\tau$-m-irreducible, so we have $a \sim a_i$ and hence $R$ strongly associate implies $a \approx a_i$ as desired.
\\
\indent(3) Let $a$ be a $\tau$-strongly irreducible element.  Suppose $a=\lambda a_1 \cdots a_n$ is a $\tau$-factorization.  Since $a$ is $\tau$-strongly irreducible, $a\approx a_i \Rightarrow a \sim a_i$ for some $i$, showing $a$ is $\tau$-irreducible as desired.
\\
\indent The proofs of (4) and (5) are immediate from definitions.
\end{proof}
\begin{theorem} Let $R$ be a pr\'esimplifiable commutative ring and $\tau$ a relation on $R^{\#}$.  Then $\tau$-irreducible, $\tau$-strongly irreducible, $\tau$-m-irreducible and $\tau$-very strongly irreducible are equivalent.
\end{theorem}
\begin{proof}Let $a\in R$ be a non-unit with $a$ $\tau$-irreducible.  If $R$ is pr\'esimplifiable, then $a\cong a$ for all $a\in R$.  Let $a\cong a_1 \cdots a_n$ with $a_i \tau a_j$ for all $i \neq j$, then $a=\lambda a_1 \cdots a_n$ for some $\lambda \in U(R)$ is a $\tau$-factorization of $a$.  Because $a$ is $\tau$-irreducible, we know $a\sim a_i$ for some $i$.  Therefore $a\cong a_i$ for some $i$, proving $a$ is $\tau$-very strongly irreducible as desired.
\end{proof}
When $R$ is a domain, all the types of irreducibles coincide and for non-zero elements, our definitions match the $\tau$-irreducible elements defined in \cite{Frazier}.  Furthermore, when we set $\tau=R^{\#} \times R^{\#}$, we get the usual factorization in integral domains for non-zero elements.  In domains, $0$ has no non-trivial factorizations anyway, so this is not much of an impediment.

\section{$\tau$-Finite Factorization Conditions on Rings with Zero-Divisors}
Let $\alpha \in \{$atomic, strongly atomic, m-atomic, very strongly atomic$ \}$, $\beta \in \{$associate, strong associate, very strong associate$\}$ and $\tau$ a symmetric relation on $R^{\#}$.  Then $R$ is said to be \emph{$\tau$-$\alpha$} if every non-unit $a\in R$ has a $\tau$-factorization $a=\lambda a_1\cdots a_n$ with $a_i$ being $\tau$-$\alpha$ for all $1\leq i \leq n$.  We will call such a factorization a \emph{$\tau$-$\alpha$-factorization}.  We say $R$ satisfies \emph{$\tau$-ACCP} if for every chain $(a_0) \subseteq (a_1) \subseteq \cdots \subseteq (a_i) \subseteq \cdots$ with $a_{i+1} \mid_{\tau} a_i$, there exists an $N\in \N$ such that $(a_i)=(a_N)$ for all $i>N$.
\\
\indent A ring $R$ is said to be a \emph{$\tau$-$\alpha$-$\beta$-UFR} if (1) $R$ is $\tau$-$\alpha$ and (2) for every non-unit $a \in R$ any two $\tau$-$\alpha$ factorizations $a=\lambda_1 a_1 \cdots a_n = \lambda_2 b_1 \cdots b_m$ have $m=n$ and there is a rearrangement so that $a_i$ and $b_i$ are $\beta$.  A ring $R$ is said to be a \emph{$\tau$-$\alpha$-HFR} if (1) $R$ is $\tau$-$\alpha$ and (2) for every non-unit $a \in R$ any two $\tau$-$\alpha$-factorizations have the same length.  A ring $R$ is said to be a \emph{$\tau$-BFR} if for every non-unit $a \in R$, there exists a natural number $N(a)$ such that for any $\tau$-factorization $a=\lambda a_1 \cdots a_n$, $n \leq N(a)$. A ring $R$ is said to be a \emph{$\tau$-$\beta$-FFR} if for every non-unit $a \in R$ there are only a finite number of non-trivial $\tau$-factorizations up to rearrangement and $\beta$.  A ring $R$ is said to be a \emph{$\tau$-$\beta$-WFFR} if for every non-unit $a \in R$, there are only finitely many $b\in R$ such that $b$ is a non-trivial $\tau$-divisor of $a$ up to $\beta$.  A ring $R$ is said to be a \emph{$\tau$-$\alpha$-$\beta$-divisor finite (df)} if for every non-unit $a \in R$, there are only finitely many $\tau$-$\alpha$ $\tau$-divisors of $a$ up to $\beta$.

\begin{theorem} \label{thm: ff props} Let $R$ be a commutative ring and $\tau$ a relation on $R^{\#}$.  Let $\alpha \in \{$atomic, strongly atomic, m-atomic, very strongly atomic$ \}$, $\beta \in \{$associate, strong associate, very strong associate$\}$ and $\tau$ a symmetric relation on $R^{\#}$. We have the following.
\\(1) $R$ is a $\tau$-$\alpha$-$\beta$-UFR implies $R$ is a $\tau$-$\alpha$-HFR.
\\(2) For $\tau$ refinable and associate preserving $R$ is a $\tau$-$\alpha$-HFR implies $R$ is a $\tau$-BFR.
\\(3) For $\tau$ refinable and associate preserving, $R$ is a $\tau$-$\alpha$-$\beta$-UFR implies $R$ is a $\tau$-$\beta$-FFR .
\\(4) $R$ is a $\tau$-$\beta$-FFR implies $R$ is a $\tau$-BFR.
\\(5) $R$ is a $\tau$-$\beta$-FFR implies $R$ is a $\tau$-$\beta$-WFFR and $R$ is a $\tau$-$\beta$-WFFR implies $R$ is a $\tau$-$\alpha$-$\beta$-df ring.
\\(6) For $\tau$-refinable and associate preserving, $R$ is a $\tau$-$\alpha$-WFFR implies $R$ is a $\tau$-$\alpha$ $\tau$-$\alpha$-$\beta$-df ring.
\\(7) $R$ is a $\tau$-$\alpha$ $\tau$-$\alpha$-$\beta$-df ring implies $R$ is a $\tau$-$\alpha$-$\beta$-df ring.
\\(8) For $\tau$ refinable and associate preserving, $R$ is a $\tau$-BFR implies $R$ satisfies $\tau$-ACCP.
\\(9) For $\tau$ refinable and associate preserving, $R$ satisfies $\tau$-ACCP implies $R$ is $\tau$-$\alpha$.
\\(10) $R$ satisfying ACCP implies $R$ satisfies $\tau$-ACCP.
\\We have the following diagram ($\star$ represents $\tau$ being refinable and associate preserving).
$$\xymatrix{
            &        \tau\text{-}\alpha \text{-HFR} \ar@{=>}^{\star}[dr]     &             &                  &                 \\
\tau\text{-}\alpha\text{-} \beta \text{-UFR} \ar@{=>}[ur] \ar@{=>}^{\star}[r]  & \tau\text{-}\beta \text{-FFR} \ar@{=>}[r] \ar@{=>}[d]  & \tau\text{-BFR} \ar@{=>}[r]^{\star}& \tau\text{-ACCP} \ar@{=>}^{\star}[r]& \tau\text{-}\alpha\\
            & \tau\text{-}\beta\text{-WFFR} \ar@{=>}[d] \ar@{=>}[dl]_{\star}                    &              &  \text{ACCP} \ar@{=>}[u]               &                  \\
\tau\text{-}\alpha\  \tau\text{-}\alpha\text{-}\beta \text{-df ring} \ar@{=>}[r]           & \tau\text{-}\alpha\text{-}\beta \text{-df ring} &
            }$$
\end{theorem}
\begin{proof} (1) Let $R$ be a $\tau$-$\alpha$-$\beta$-UFR.  Then we have $R$ is $\tau$-$\alpha$ and every $\tau$-$\alpha$-factorization of any non-unit $a\in R$ has the same length, so $R$ is a $\tau$-$\alpha$-HFR.
\\
\indent(2) Let $\tau$ be refinable and associate preserving, with $R$ a $\tau$-$\alpha$-HFR.  Let $a\in R$ be a non-unit, and $a=\lambda a_1 \cdots a_n$ be a $\tau$-$\alpha$-factorization of $a$.  Set $N(a)=n$.  Suppose there were a $\tau$-factorization of $a$, of length $m>n$, $a=\mu b_1 \cdots b_m$.  This can be $\tau$-refined into a $\tau$-$\alpha$-factorization since $R$ is $\tau$-$\alpha$ and $\tau$ is refinable and associate preserving.  This would lead to a strictly longer $\tau$-$\alpha$-factorization of $a$ contradicting the fact that $R$ is a $\tau$-$\alpha$-HFR.
\\
\indent(3) Let $R$ be a $\tau$-$\alpha$-$\beta$-UFR, with $\tau$ refinable and associate preserving.  Let $a\in R$ be a non-unit.  Say $a=\lambda a_1 \cdots a_n$ is the unique $\tau$-$\alpha$ factorization up to rearrangement and $\beta$.  For any $\tau$-factorization $a=\mu b_1 \cdots b_m$, take the unique $\tau$-$\alpha$-factorization of each $b_i$, say $b_i=\mu_ic_{i_1} \cdots c_{i_{mi}}$.  We may now refine our $\tau$-factorization to be

$$a=\mu (\mu_1c_{1_1} \cdots c_{1_{m1}})(\mu_2c_{2_1} \cdots c_{2_{m2}}) \cdots (\mu_mc_{m_1} \cdots c_{m_{mm}})$$
$$=(\mu\mu_1\mu_2 \cdots \mu_m) c_{1_1} \cdots c_{1_{m1}}c_{2_1} \cdots c_{2_{m2}} \cdots c_{m_1} \cdots c_{m_{mm}}$$

This is a $\tau$-$\alpha$-factorization of $a$, so there is a rearrangement such that $c_i$ and $a_i$ are $\beta$.  This means any $\tau$-factorization of $a$ is simply some grouping of $\beta$ of the $a_i$ in the original $\tau$-$\alpha$-factorization of $a$.  There are only $2^{n}$ possible ways to do this up to $\beta$, so $R$ is a $\tau$-$\beta$-FFR.
\\
\indent (4) Let $R$ be a $\tau$-$\beta$-FFR, with $a\in R$ a non-unit.  There are only finitely many $\tau$-factorizations of $a$ up to $\beta$.  Simply set $N(a)$ to the maximum length of any of these $\tau$-factorizations.
\\
\indent(5) Let $R$ be a $\tau$-$\alpha$-FFR with $a$ a non-unit $a\in R$.  We collect each of the $\tau$-factors in the finite number of $\tau$-factorizations up to $\beta$.  This is a complete list of non-trivial $\tau$-divisors of $a$ up to $\beta$.  Moreover, it is a finite union of finite sets, hence is finite.  This proves $R$ is a $\tau$-$\beta$-WFFR.  Every $\tau$-$\alpha$-divisor is certainly a $\tau$-divisor, so the second implication is immediate.
\\
\indent(6) Let $R$ be a $\tau$-$\beta$-WFFR with $\tau$ refinable and associate preserving.  We have just seen that $R$ is a $\tau$-$\alpha$-$\beta$-df ring, so we need only show $R$ is $\tau$-$\alpha$.  In light of (9), it suffices to show that $R$ satisfies $\tau$-ACCP.  We suppose for a moment there is an infinite ascending chain of properly contained principal ideals $(a_0) \subsetneq (a_1) \subsetneq \cdots$ with $a_{i+1}\mid_{\tau} a_i$.  Say $a_i=\lambda_i a_{i+1}b_{i_1}\cdots b_{i_{n_i}}$ for each $i$.  We must have $n_i\geq 1$ for all $i$ otherwise $(a_i)=(a_{i+1})$.  Using the fact that $\tau$ is refinable and associate preserving, we know that we have the following $\tau$-factorizations of $a$:
$$a_0=\lambda_0 a_{1}b_{0_1}\cdots b_{0_{n_0}}=\lambda_0 (\lambda_1 a_{2}b_{1_1}\cdots b_{1_{n1}})b_{01}\cdots b_{0n_0}=$$
$$(\lambda_0 \lambda_1 \lambda_2) a_{3}b_{2_1}\cdots b_{2_{ni}}b_{1_1}\cdots b_{1_{n1}}b_{0_1}\cdots b_{0_{n0}}=...$$
So in particular, for $i>0$, each $a_i$ is a $\tau$-divisor of $a_0$.  Furthermore, none are even associate, so certainly none are $\beta$. Hence $a_0$ has an infinite number of $\tau$-divisors up to $\beta$.  This contradicts the hypothesis that $R$ is a $\tau$-$\alpha$-WFFR.
\\
\indent(7) This is immediate from definitions.
\\
\indent (8) Let $\tau$ be refinable and associate preserving and $R$ a $\tau$-BFR.  Suppose $(a_0)\subsetneq (a_1) \subsetneq \cdots \subsetneq (a_i) \subsetneq \cdots$ is an infinite chain of properly ascending principal ideals such that $a_{i+1} \mid_{\tau} a_i$ for each $i$.  Then we use the same factorization as in (6) to see that we get arbitrarily long $\tau$-factorizations of $a_0$ contradicting the hypothesis.
\\
\indent(9) Suppose $R$ satisfies $\tau$-ACCP, and $\tau$ is refinable and associate preserving.  Let $a\in R$ be a non-unit.  We show $a$ has a $\tau$-$\alpha$ factorization.  If $a$ is $\tau$-$\alpha$, we are done, so we may assume $a=\lambda_1 a_{1_1} \cdots a_{1_{n_1}}$ is a non-trivial $\tau$-factorization with $a$ and $a_{1_i}$ not $\beta$ for all $1 \leq i \leq n_1$.  If all of the $a_{i_1}$ are $\tau$-$\alpha$, we are done as we have found a $\tau$-$\alpha$ factorization of $a$.  So at least one must not be $\tau$-$\alpha$, say it is $a_{1_1}$, so suppose $a_{1_1}=\lambda_2 a_{2_1} \cdots a_{2_{n_2}}$ is a non-trivial $\tau$-factorization with $a_{1_1}$ and $a_{2_i}$ not $\beta$ for all $1 \leq i \leq n_2$.  Then we have $a=(\lambda_1 \lambda_2) a_{2_1} \cdots a_{2_{n_2}}a_{1_1} \cdots a_{1_{n_1}}$ is a $\tau$-factorization.  We could continue in the fashion picking out one factor that is not $\tau$-$\alpha$, always just saying it is $a_{i_1}$ after reordering if neces
 sary.  This yields an infinite chain of principal ideals $(a) \subsetneq (a_{1_1}) \subsetneq (a_{2_1}) \subsetneq \cdots$ with $a_{{i+1}_1} \mid_{\tau} a_{{i}_1}$ which contradicts $R$ satisfying $\tau$-ACCP.
\\
\indent(10) This is clear by noting that if $a\mid_{\tau} b$, then $a\mid b$.  If $R$ failed to satisfy $\tau$-ACCP, there would be a properly ascending infinite chain of principal ideals $(a_0)\subsetneq (a_1) \subsetneq \cdots$ with $a_{i+1}\mid_{\tau} a_i$ also satisfies $a_{i+1}\mid a_i$.  Hence we would have an infinite chain of properly ascending principal ideals which contradicts ACCP.
\end{proof}

\section{The relation $a \tau_z b \Leftrightarrow ab=0$}
\label{subsec: tau-z}
\indent Let $a,b \in R^{\#}$.  We will consider the relation $\tau_z$ defined by $a\tau_z b$ if and only if $ab=0$.  We will analyze the relation $\tau_z$ and investigate rings satisfying the $\tau_z$-finite factorization properties described in Section 4.
\\
\indent We observe that with the exception of nilpotent elements, we have a strong correspondence between $\tau_z$-factorizations and the zero-divisor graphs studied first by Beck in \cite{Beck} and then by several more authors in particular in \cite{Andersonzdg, Davidanderson, Livingston}.  The zero-divisor graph, denoted $\Gamma(R)$, is defined to be the graph with vertex set $Z(R)-\{0\}$.  Edges given by the relationship $a,b\in Z(R)-\{0\}$ are adjacent if $ab=0$.  So we see $a\tau_z b\Leftrightarrow ab=0 \Leftrightarrow $ $a$ and $b$ are adjacent in $\Gamma(R)$ or $a=b$ with $a^2=0$.  We would like to be able to say $a \tau_z b$ if and only if $a$ and $b$ are adjacent in $\Gamma(R)$.
\\
\indent There are two approaches to ensuring this can be said: (1) insist that our ring $R$ is reduced so there are no non-trivial nilpotent elements or (2) define a modification of $\tau_z$ to be $\tau_z^{\Delta}:=\tau_z - \Delta \cap \left(\text{Nil}(R)\times \text{Nil}(R)\right)$, that is $a \tau_z^{\Delta} b \Leftrightarrow ab=0 \text{ and } a\neq b$.  Both of these choices result in having no repeated factors in any given $\tau_z^{\Delta}$ ($\tau_z$)-factorization (in a reduced ring) which will be useful in several of the proofs.
\begin{theorem}\label{thm: tau props}Let $R$ be a commutative ring and $\tau_z$ and $\tau_z^{\Delta}$ be as defined above.
\\(1) For $a\in R^{\#}$, $a$ has only trivial $\tau_z^{\Delta}$($\tau_z$)-factorizations and therefore is a $\tau_z^{\Delta}$($\tau_z$)-atom.
\\(2) $\tau_z^{\Delta}$($\tau_z$) is symmetric, but not combinable, and therefore not multiplicative.  Furthermore, $\tau_z^{\Delta}$ ($\tau_z$) is refinable, but is not divisive.
\\(3) $R$ satisfies $\tau_z^{\Delta}$$(\tau_z)$-ACCP.
\\(4) $R$ is $\tau_z^{\Delta}$($\tau_z$)-atomic.
\\(5) If $R$ is an integral domain, then $R$ is a $\tau_z^{\Delta}$($\tau_z$)-atomic-associate-UFR.
\\(6) $\tau_z$ is associate (resp. strongly associate, resp. very strongly associate) preserving, while $\tau_z^{\Delta}$ is not.
\end{theorem}
\begin{proof}
(1) Let $a\in R^{\#}$.  Suppose $a=\lambda a_1 \cdots a_n$ is a $\tau_z^{\Delta}$ ($\tau_z$)-factorization.  If $n\geq 2$, then $a_1 \cdot a_2 =0$, so $a=0$, a contradiction.  Hence, $n=1$ and there are only trivial $\tau_z^{\Delta}$($\tau_z$)-factorizations.
\\
\indent (2) $\tau_z^{\Delta}$($\tau_z$) is clearly symmetric.  Let $R=\Z/30\Z$ and consider $0=6\cdot 10 \cdot 15$ is a $\tau_z^{\Delta}$($\tau_z$)-factorization, but $0=6\cdot 150=6\cdot 0$ is not a $\tau_z^{\Delta}$ ($\tau_z$)-factorization. This shows $\tau_z^{\Delta}$($\tau_z$) is not combinable, and hence not multiplicative.   Now let $R=\Z/12\Z$, we have $2 \tau_z^{\Delta}(\tau_z) 6$, but $2\not \tau_z^{\Delta}(\not \tau_z) 3$, so $\tau_z^{\Delta}(\tau_z)$ is not divisive.  Every non-trivial $\tau_z^{\Delta}$ ($\tau_z$)-factor is non-zero and in light of (1) has no non-trivial $\tau_z^{\Delta}$ ($\tau_z$)-factorizations so $\tau_z^{\Delta}$ ($\tau_z$) is vacuously refinable.
\\
\indent(3) Let $(a)\subseteq (b)$ with $b\mid_{\tau_z^{\Delta}}a$ ($b\mid_{\tau_z}a$), say $a=\lambda b b_1 \cdots b_n$ is a $\tau_z^{\Delta}$ ($\tau_z$)-factorization.  If $n \geq 1$, $b\tau_z^{\Delta} (\tau_z) b_1 \Rightarrow a=0$.  If $n=0$, then $(a)=(b)$.  Hence, the longest $\tau_z^{\Delta} (\tau_z)$-ascending chain has length 1.
\\
\indent(4) We have already seen that all non-zero, non-units are $\tau_z^{\Delta}$ ($\tau_z$)-atoms from (1).  If $Z(R)=0$, then $0$ has only trivial factorizations, making it a $\tau_z^{\Delta}$ ($\tau_z$)-atom.  Suppose $R$ is not a domain.  Choose an $x \in Z(R)$ such that there is a $y\in R$ such that $xy=0$ for $x,y \neq 0$ and $x \neq y$.  Then $0=xy$ is a $\tau_z^{\Delta}$ ($\tau_z$)-atomic factorization of $0$.  If it is not possible to choose such an $x$, then $x^2=0$ for every $0\neq x \in Z(R)$.  This means $0$ itself is a $\tau_z^{\Delta}$-atom ($0=x\cdot x$ is a $\tau_z$-atomic factorization).
\\
\indent(5) If $R$ is a domain, then $Z(R)=0$ and we have $\tau_z^{\Delta} (\tau_z)=\emptyset$.  There are only trivial $\tau_z^{\Delta}$ ($\tau_z$)-factorizations, so every non-unit is a $\tau_z^{\Delta}$ ($\tau_z$)-atom, and so $R$ is a $\tau_z^{\Delta}$ ($\tau_z$)-atomic-associate-UFR.
\\
\indent(6) Suppose $a\tau_z b$, with $a \sim a'$ (resp. $a\approx a'$, $a\cong a'$).  Then in all cases, we have $(a)=(a')$ and therefore $a'=ra$ for some $r\in R$. We have $ab=0$, but by substitution, we have $a'b=(ra)b=r(ab)=0$, so $a' \tau_z b$ as well.  This shows $\tau_z$ is associate (resp. strongly associate, very strongly associate) preserving.  On the other hand, let $R=\Z/9\Z$.  $3\sim 6$ (resp. $3\approx 6$, $3 \cong 6$) and $3 \tau_z^{\Delta} 6$; however, $3 \not \tau_z^{\Delta} 3$.  Thus $\tau_z^{\Delta}$ is not associate (resp. strongly associate, very strongly associate) preserving.
\end{proof}
We begin by stating a theorem which summarizes some results about zero-divisor graphs.  We denote the complete graph on $r$ vertices with $K^r$, and define $\omega(\Gamma(R))$ to be the clique number of $\Gamma(R)$.  This is the largest integer $r \geq 1$ with $K^r \subseteq \Gamma(R)$.  If $K^r \subseteq \Gamma(R)$ for all $r\geq 1$, then we say $\omega(\Gamma(R))=\infty$.  We use min$(R)$ to denote the set of minimal prime ideals of $R$.
\begin{theorem}\label{thm: ZG results} (Zero-divisor graph results) Let $R$ be a commutative ring.
\\(1) $\Gamma(R)$ is connected and has diameter less than or equal to $3$.
\\(2) $\Gamma(R)$ is finite if and only if $R$ is a domain or $R$ is finite.
\\(3) $\omega(\Gamma(R))=\infty$ if and only if $\Gamma(R)$ has an infinite clique (a complete subgraph).
\\(4) $\omega(\Gamma(R)) < \infty$ if and only if $|$Nil$(R)|<\infty$ and Nil$(R)$ is a finite intersection of primes, that is $|\text{min}(R)| < \infty$.
\\(5) All rings with $|\Gamma(R)|\leq 4$ have been classified up to isomorphism.
\\(6) All finite rings with $|\omega(\Gamma(R))|\leq 3$ have been classified up to isomorphism.
\\(7) If $R=R_1 \times \cdots \times R_n$ with $R_i$ domains, $n \geq 2$, $\omega(\Gamma(R))=n$.
\end{theorem}
\begin{proof}
(1) \cite[Theorem 2.3]{Livingston}.  (2) \cite[Theorem 2.2]{Livingston}.  (3) and (4) \cite[Theorem 3.7]{Beck}.  (5) \cite[Example 2.1]{Livingston}.  (6) \cite[Page 226]{Beck} and \cite[Theorem 4.4]{Andersonzdg}.  (7) \cite[Theorem 3.7]{Davidanderson}.
\end{proof}
\begin{theorem}\label{thm: wffr} Let $R$ be a commutative ring.  The following are equivalent.
\\(1) $\mid\Gamma(R)\mid < \infty$.
\\(2) $R$ is a domain or $R$ is finite.
\\(3) $R$ is a strong-$\tau_z^{\Delta}$-FFR (for every non-unit $a\in R$, there are only a finite number of non-trivial $\tau_z^{\Delta}$-factorizations of $a$).
\\(4) $R$ is a strong-$\tau_z^{\Delta} (\tau_z)$-WFFR (for every non-unit $a\in R$, there are only a finite number of non-trivial $\tau_z^{\Delta} (\tau_z)$-divisors of $a$).
\\(5) $R$ is a strong-$\tau_z^{\Delta} (\tau_z)$-atomic-divisor finite-ring (for every non-unit $a\in R$, there are only a finite number $\tau_z^{\Delta} (\tau_z)$-divisors which are $\tau_z^{\Delta} (\tau_z)$-atoms).
\\(6) $\Gamma(R)$ has a finite number of complete subgraphs $K^r$ for $r\geq 2$.
\end{theorem}
\begin{proof}
(1) $\Leftrightarrow$ (2) This is given by Theorem \ref{thm: ZG results} (2).
\\
\indent(2) $\Rightarrow$ $\left[(3) \text{ and } (4) \right]$ For $R$ a domain, the result is trivial as every non-unit is a $\tau_z^{\Delta} (\tau_z)$-atom.  For $R$ finite, say $|R|= n$ we see that (4) clearly holds as there are only $n$ possible $\tau_z^{\Delta} (\tau_z)$-divisors.  Furthermore, because no factor can be repeated in a $\tau_z^{\Delta}$-factorization, there are at most $2^n$ possible $\tau_z^{\Delta}$-factorizations.
\\
\indent $\left[(3) \text{ or } (4) \right]$ $\Rightarrow$ (1) Suppose $|\Gamma(R)|$ is infinite.  The $\Gamma(R)$ has an infinite number of distinct vertices; say $\{x_i\}_{i=1}^{\infty}$.  Recall, $\Gamma(R)$ is connected, so every vertex is adjacent to another distinct vertex, say $y_i$ for each $i$.  Then $\{x_iy_i\}$ is an infinite collection of $\tau_z^{\Delta} (\tau_z$)-factorizations of $0$ up to reordering.  This contradicts (3).  Each $x_i$ is a distinct non-trivial $\tau_z^{\Delta} (\tau_z)$-divisor of $0$ which contradicts (4).
\\
\indent (4) $\Leftrightarrow$ (5) Every $\tau_z^{\Delta} (\tau_z)$-divisor is non-zero and hence is $\tau_z^{\Delta} (\tau_z)$-atomic.  Every $\tau_z^{\Delta} (\tau_z)$-atomic $\tau_z^{\Delta} (\tau_z)$-divisor is certainly a $\tau_z^{\Delta} (\tau_z)$-divisor.
\\
\indent (3) $\Rightarrow$ (6) Suppose there were an infinite number of distinct complete subgraphs in $\Gamma(R)$ of size at least 2.  Each subgraph corresponds to a distinct non-trivial $\tau_z^{\Delta} (\tau_z)$-factorization of $0$ by taking the product of the vertices in the given complete subgraph, contradicting (3).
\\
\indent(6) $\Rightarrow$ (1) Suppose for a moment $|\Gamma(R)|$ were infinite.  Let $\{x_i\}_{i=1}^{\infty}$ be an infinite set of distinct vertices.  Recall, $\Gamma(R)$ is connected, so every vertex $x_i$ must be adjacent to another vertex, say $y_i$.  Then $x_i$ and $y_i$ form a complete subgraph of size $2$, and this generates an infinite collection, contradicting (6).
\end{proof}
\begin{theorem}\label{thm: ffr} Let $R$ be a commutative ring.  The following are equivalent.
\\(1) $\text{Nil}(R)=0$ and $\mid\Gamma(R)\mid<\infty$.
\\(2) $R$ is a strong-$\tau_z$-FFR (for every non-unit $a\in R$, there are only a finite number of non-trivial $\tau_z$-factorizations).
\\(3) $R$ is a domain or a finite reduced ring.
\\(4) $R$ is a domain or $R\cong K_1 \times \cdots \times K_n$ with $K_i$ a finite field for $1\leq i \leq n$ with $n\geq 2$.
\\(5) Nil$(R)=0$ and $\Gamma(R)$ has a finite number of complete subgraphs $K^r$ with $r \geq 2$.
\end{theorem}
\begin{proof}
$(1) \Rightarrow (2)$ Suppose there are an infinite number of non-trivial $\tau_z$-factorizations of $0$.  All $\tau_z$-factors are distinct.  If $a$ were a repeated $\tau_z$ factor of $0$, then $a\tau_z a \Rightarrow a^2=0$ implies $0\neq a \in \text{Nil}(R)$, a contradiction.  Hence, there can only be at most $2^{|\Gamma(R)|}$ non-trivial $\tau_z$-factorizations, a contradiction.
\\
\indent(2) $\Rightarrow$ (5) Suppose there were an infinite number of distinct complete subgraphs in $\Gamma(R)$ of size at least 2.  Each such complete subgraph corresponds to a distinct non-trivial $\tau_z$-factorization of $0$ contradicting (2).  Suppose $R$ were not reduced.  Suppose $0\neq x \in \text{Nil}(R)$, with $x^k=0$ with $k$ minimal.  Then $0=(x^{k-1})^i$ is a $\tau_z$ factorization of $0$ of length $i$ for any $i\geq 2$.  Hence $R$ must be reduced.
\\
\indent(5) $\Rightarrow (1)$ Suppose there were an infinite number of vertices in $\Gamma(R)$.  We recall that $\Gamma(R)$ is connected.  We could find paths connecting all the vertices.  This would certainly require an infinite number of edges.  This yields an infinite number of $K^2$ subgraphs, contradicting (5).
\\
\indent$(1) \Leftrightarrow (3)$  We have now added the hypothesis that Nil$(R)=0$ to both (1) and (2) of Theorem \ref{thm: wffr}, so the equivalence remains.
\\
\indent(3) $\Leftrightarrow$ (4) This is well known.
\end{proof}
We now introduce the notion of the associated zero-divisor graph, $\Gamma(R/ \sim)$.  The vertices are now represented by a zero-divisor up to associate, and an edge between two zero-divisor representatives $a$ and $b$ if $ab=0$.  Recall $\sim$ is an equivalence relation, and one can check the edge relation is well defined.  We record two analogous theorems, but omit the proofs.  The proofs are nearly identical to those of Theorems \ref{thm: wffr} and \ref{thm: ffr} except now uniqueness is only up to associate and reordering.
\begin{theorem}\label{thm: zd ass ffr}Let $R$ be a commutative ring.  The following are equivalent.
\\(1) $|\Gamma(R/\sim)| < \infty$ (There are a finite number of zero-divisors up to associate).
\\(2) $R$ is a $\tau_z^{\Delta}$-associate-FFR.
\\(3) $R$ is a $\tau_z^{\Delta} (\tau_z)$-associate-WFFR.
\\(4) $R$ is a $\tau_z^{\Delta} (\tau_z)$-atomic-associate-divisor finite ring.
\\(5) $\Gamma(R/\sim)$ has a finite number of complete subgraphs $K^r$ for $r\geq 2$.
\end{theorem}
\begin{theorem}\label{thm: z ass ffr} Let $R$ be a commutative ring.  The following are equivalent.
\\(1) $R$ is a $\tau_z$-associate-FFR.
\\(2) Nil$(R)=0$ and $\mid \Gamma(R/\sim) \mid < \infty$.
\\(3) Nil$(R)=0$ and $\Gamma(R/\sim)$ has a finite number of complete subgraphs $K^r$ for $r \geq 2$.
\end{theorem}
\begin{example}
\end{example}
Consider $R=\Z/4\Z$.  We have $\tau_z=\{(2,2)\}$, while $\tau_z^{\Delta}=\emptyset$.  $0=2^i$ is a $\tau_z$-factorization for all $i\geq 2$, so $R$ is not a (strong-)$\tau_z$-associate-FFR; however, $R$ is certainly a (strong)-$\tau_z^{\Delta}$-associate-FFR since there are only trivial factorizations.  Hence the items given in Theorems \ref{thm: wffr} and \ref{thm: ffr} (resp. \ref{thm: z ass ffr} and \ref{thm: zd ass ffr}) cannot be combined and still maintain equivalence.

\begin{theorem} \label{thm: tau bfr}Let $R$ be a commutative ring.  Then $R$ is a $\tau_z^{\Delta}(\tau_z)$-BFR if and only if ($R$ is reduced and) $\omega(\Gamma(R))$ is finite.\end{theorem}
\begin{proof} All $\tau_z^{\Delta}(\tau_z)$-factorizations of non-zero elements are trivial, and hence length 1.  Let $0=\lambda a_1 \cdots a_n$ be a $\tau_z^{\Delta}(\tau_z)$-factorization.  Now $a_ia_j=0$ for all $i\neq j$, ($R$ being reduced tells us) $a_i \neq a_j$ for all $i \neq j$.  Hence, every non-trivial $\tau_z^{\Delta}(\tau_z)$-factorization corresponds precisely with a complete subgraph of $\Gamma(R)$.
\\
\indent($\Rightarrow$)If $R$ is a $\tau_z^{\Delta}(\tau_z)$-BFR, then there is a bound on the maximum length of any $\tau_z^{\Delta}(\tau_z)$-factorization of $0$, say $n$.  There can be no complete subgraph of size larger than $K^n$.  (Furthermore, suppose $0\neq x\in \text{Nil}(R)$ with $x^k=0$, the smallest such integer $k$, then $0=(x^{k-1})^i$ for $i \geq 2$ yields arbitrarily long $\tau_z$-factorizations, a contradiction, so $R$ is reduced.)
\\
\indent($\Leftarrow$) Conversely, if we assume $\omega(\Gamma(R))=n$ (, with $R$ reduced).  Then all of the $\tau_z^{\Delta}(\tau_z)$-factorizations of $0$ are bounded by $n$.  All $\tau_z^{\Delta}$ $(\tau_z)$ factorizations of non-zero elements are of length 1.  Hence $R$ is a $\tau_z^{\Delta}(\tau_z)$-BFR as desired.
\end{proof}
\begin{example} We can construct a $\tau_z^{\Delta}(\tau_z)$-BFR that has a factorization of length $n$ and no longer for any $n\geq 1$.
\end{example}
Consider the ring $R=K_1 \times \cdots \times K_n$ with $K_i$ a field for $1\leq i \leq n$.  By (7) of Theorem \ref{thm: ZG results}, we have $\omega(\Gamma(R))=n$.  This gives us a complete subgraph of size $n$ with vertices $\{x_1, \ldots ,x_n\}$ which corresponds to a $\tau_z^{\Delta} (\tau_z)$-factorization $0=x_1 \cdots x_n$.  There can be no longer factorizations or else there would be a complete subgraph of size larger than $n$.  One can simply take the standard basis $x_i=e_i:=(0_{K_1},\cdots, 1_{K_i}, \cdots, 0_{K_n})$ where the $1$ occurs in the $i^{\text{th}}$ coordinate for $1 \leq i \leq n$.
\begin{corollary}Let $R$ be a commutative ring.  $R$ is a $\tau_z^{\Delta}(\tau_z)$-BFR if and only if ($R$ is reduced and) $\text{Nil}(R)$ is finite and $\text{Nil}(R)$ is a finite intersection of prime ideals, i.e. min$(R)$ is finite. \end{corollary}
\begin{proof} This is a consequence of (4) from Theorem \ref{thm: ZG results} and the above theorem.  (For a reduced ring $\text{Nil}(R)=0$ which is certainly finite.)  If min$(R)$ is finite, then $\omega(\Gamma(R))$ is finite, and hence by the above theorem, $R$ is a $\tau_z^{\Delta}(\tau_z)$-BFR.  Conversely, if $R$ is a $\tau_z^{\Delta}(\tau_z)$-BFR, from the theorem, we know ($R$ must be reduced and) that $\omega(\Gamma(R))$ is finite.  Therefore, min$(R)$ is finite proving the claim.
\end{proof}
\begin{corollary}Any (reduced) Noetherian ring with $\text{Nil}(R)$ finite, or more generally any (reduced) ring with $\text{Nil}(R)$ finite that satisfies the ascending chain condition on radical ideals is a $\tau_z^{\Delta}(\tau_z)$-BFR.
\end{corollary}
\begin{proof} This is a consequence of \cite[Theorem 87]{Kaplansky} and the fact that ($R$ being reduced yields $\sqrt{0}=\text{Nil}(R)=0$) $\text{Nil}(R)$ is a radical ideal and hence a finite intersection of primes, with $\text{Nil}(R)$ finite.
\end{proof}
\begin{theorem}\label{thm: tau hfr}Let $R$ be a commutative ring.  Then $R$ is a $\tau_z^{\Delta}(\tau_z)$-atomic-HFR if and only if ($R$ is reduced and) $|\omega(\Gamma(R))|\leq 2$. \end{theorem}
\begin{proof}
($\Rightarrow$) Let $R$ be a $\tau_z^{\Delta}(\tau_z)$-atomic-HFR.  (If $R$ is not reduced, it is not even a $\tau_z$-BFR, so this is impossible.)  Suppose $|\omega(\Gamma(R))|>2$, then there is a $K^3 \subset \Gamma(R)$, say $ab=0$, $ac=0$ and $bc=0$ with $a,b,c \in Z(R)$ all distinct.  The $\tau_z^{\Delta}(\tau_z)$-factorizations $0=ab=abc$ show that $R$ cannot be a $\tau_z^{\Delta}(\tau_z)$-atomic-HFR, a contradiction.
\\
\indent($\Leftarrow$) Let $|\omega(\Gamma(R))|\leq 2$.  Recall $R$ is always $\tau_z^{\Delta}(\tau_z)$-atomic.  All non-zero elements have only trivial $\tau_z^{\Delta}(\tau_z)$-factorizations, and hence have the same length, $1$.  If $|\omega(\Gamma(R))|= 0$, then $R$ is a domain and $R$ is even a $\tau_z^{\Delta}(\tau_z)$-atomic-associate-UFR.  If $|\omega(\Gamma(R))|= 1$, then there is only one possible non-trivial $\tau_z^{\Delta}(\tau_z)$-factor.  Our hypothesis implies that any $\tau_z^{\Delta}(\tau_z)$-factorization cannot have a repeated factor.  Hence $0$ has only trivial $\tau_z^{\Delta}(\tau_z)$-factorizations.  If $\omega(\Gamma(R))= 2$, then $0$ is not a $\tau_z^{\Delta}(\tau_z)$-atom, so $\tau_z^{\Delta}(\tau_z)$-atomic-factorizations of $0$ must be at least length 2. Furthermore, $|\omega(\Gamma(R))|\leq 2$ implies $\tau_z^{\Delta}(\tau_z)$-factorizations of $0$ have length at most 2, proving $R$ is a $\tau_z^{\Delta}(\tau_z)$-atomic-HFR.
\end{proof}
The following lemma is well known, so we omit the proof.
\begin{lemma}\label{lem: idempotent}Let $R$ be a commutative ring. Suppose $(a)=(a^2)$, then there exists $e\in R$ such that $e^2=e$ and $e\approx a$.  Furthermore, $R$ is decomposable, i.e. $R=eR\times (1-e)R=R_1\times R_2$.\end{lemma}
\begin{theorem}Let $R$ be a commutative ring.  Then $R$ is a $\tau_z$-atomic-associate-UFR if and only if $R$ is a domain or a direct product of two fields.
\end{theorem}
\begin{proof}($\Rightarrow$) $R$ being a $\tau_z$-atomic-associate-UFR implies $R$ is a $\tau_z$-atomic-HFR, and therefore by the previous result $\omega(\Gamma(R))\leq 2$.  Now $\omega(\Gamma(R))=0 \Leftrightarrow $ $R$ is a domain.  If $\omega(\Gamma(R))=1$, then there is only one non-zero zero-divisor.  As before, this is forced to be nilpotent, making $R$ not even a $\tau_z$-atomic-BFR.  The only possibility remaining is for $\omega(\Gamma(R))=2$.  Hence there is a $\tau_z$-factorization of $0$ of length 2, say $xy=0$.  $R$ is a $\tau_z$-atomic-associate-UFR, so $x$ and $y$ are the only two $\tau_z$-factors up to associate.
\\
\indent We wish to show that $R$ is decomposable.  We cannot have $x^2=0$ or $y^2=0$ since $R$ must be reduced.  All the same, $x^2$ and $y^2$ are certainly still zero-divisors.  They must be associate to either $x$ or $y$.  If $x^2 \sim x$ we have a non-trivial idempotent element and $R$ is decomposable and we are done. Thus we may assume $x^2 \sim y$ and $y^2 \sim x$.  This means $x^2 \mid y$, so certainly $x^2 \mid y^2$, and $y^2 \mid x$.  Hence we have $x^2 \mid x$ and therefore $x^2 \sim x$.
\\
\indent In all cases there is a non-trivial idempotent element $e$ and we can write $R=R_1 \times R_2$.  As in Lemma \ref{lem: idempotent}, $x \approx e$, where $e$ is identified with $(1,0)$.  Furthermore, $x$ can be identified with $(\lambda_x, 0)$ where $\lambda_x\in U(R_1)$ and $y$ can be identified with $(0, \lambda_y)$ where $\lambda_y\in U(R_2)$.  Let $0 \neq a \in R_1$, and $0\neq b \in R_2$.  We show they must be units.  Every element of $R=R_1\times R_2$ of the form $(a,0)$ or $(0,b)$ with $a,b$ non-zero is a zero-divisor.  They must be associate to either $(\lambda_x,0)$ or $(\lambda_y,0)$.  This forces $(a)=(\lambda_x)=R_1$ and $(b)=(\lambda_y)=R_2$.  Hence we must have $a\in U(R_1)$ and $b\in U(R_2)$ which means $R_1$ and $R_2$ are fields as desired.
\\
\indent ($\Leftarrow$) For domains this is immediate.  If $R=K_1 \times K_2$ for fields $K_1, K_2$, then the only non-units are of the form $(a,0)$ and $(0,b)$.  So $0$ is not a $\tau_z$-atom.  The only non-trivial $\tau_z$-factorizations are of the form $(0,0)=(a,0)(0,b)$ for $0\neq a\in K_1$, $0\neq b\in K_2$.  This is the only factorization up to rearrangement and associate, so $R$ is a $\tau_z$-atomic-associate-UFR.
\end{proof}
\begin{theorem} Let $R$ be a finite reduced commutative ring.  Then $R$ is a $\tau_z$-atomic-associate-UFR if and only if $\tau_z$-atomic-HFR.
\end{theorem}
\begin{proof} ($\Rightarrow$) This is always true. ($\Leftarrow$) $\tau_z$-atomic-HFR implies $\omega(\Gamma(R)) \leq 2$.  Any finite, reduced ring is of the form $R=K_1 \times \cdots \times K_n$ with $K_i$ finite fields.  We recall from Theorem \ref{thm: ZG results} (7) that $\omega(\Gamma(K_1 \times \cdots \times K_n))=n$.  So in fact, we must have $R=K_1$ or $R \cong K_1 \times K_2$ for some finite fields $K_1,K_2$.  Both cases are covered by the previous theorem, so $R$ is a $\tau_z$-atomic-associate-UFR.
\end{proof}
\begin{example} Let $R$ be a finite commutative ring with $1$.  Then $R$ is a $\tau_z^{\Delta}$-HFR does not imply that $R$ is a $\tau_z^{\Delta}$-UFR.
\end{example}
Consider the ring $R=K_1 \times \Z/4\Z$ (with $K_1$ a finite field).  Now $\omega(\Gamma(R))=2$ by \cite[Theorem 3.2]{Davidanderson}, so $R$ is a $\tau_z^{\Delta}$-atomic-HFR by the above theorem.  However, $(0,0)=(1,0)(0,1)=(1,2)(0,2)$ but $(0,2)\not \sim (1,0)$ and $(0,2) \not \sim (0,1)$, showing there exist non-unique $\tau_z^{\Delta}$-factorizations of $0$ in this ring.
\section*{Acknowledgements} I would like to thank Dr. Daniel D. Anderson for his support and encouragement.

{\small \hfill Submitted April 22, 2012}
\end{document}